\documentclass[a4paper,11pt]{article}

\usepackage[mathscr]{euscript}

\usepackage{amsfonts,amscd,amssymb,amsmath,amsthm,bbm}
\usepackage{centernot,caption,subcaption,color,url,tikz}
\usepackage[T1]{fontenc}
\usepackage{authblk}
\usepackage{graphicx}
\usepackage{caption}
\usepackage{subcaption}
\usepackage[colorlinks=true,linkcolor=blue,citecolor=red]{hyperref}
\usepackage{hyperref}
\usepackage[capitalize]{cleveref}

\usepackage[margin=1in]{geometry}
\usepackage{setspace}

\newtheorem{theorem}{Theorem}
\newtheorem{thm}{Theorem}[section]
\newtheorem{corollary}[thm]{Corollary}
\newtheorem{lemma}[thm]{Lemma}
\newtheorem{claim}[thm]{Claim}

\theoremstyle{remark}

\newcommand{\1}{\mathbbm{1}}

\usepackage{amsmath,amssymb,amsfonts,latexsym,epsfig,graphicx,longtable,wrapfig}

\title{Reconstructing random graphs from distance queries}

\author{\Large Michael Krivelevich\thanks{School of Mathematical Sciences, Tel Aviv University, Tel Aviv 6997801, Israel. 
\newline
Email: krivelev@tauex.tau.ac.il. 
}, \,\,\,\,\, Maksim Zhukovskii\thanks{The University of Sheffield, Department of Computer Science, Sheffield S1 4DP, UK.\newline Email: m.zhukovskii@sheffield.ac.uk.}}

\date{}

\begin{document}

\maketitle

\begin{abstract}
We estimate the minimum number of distance queries that is sufficient to reconstruct the binomial random graph $G(n,p)$ with constant diameter with high probability. We get a tight (up to a constant factor) answer for all $p>n^{-1+o(1)}$ outside ``threshold windows'' around $n^{-k/(k+1)+o(1)}$, $k\in\mathbb{Z}_{>0}$: with high probability the query complexity equals $\Theta(n^{4-d}p^{2-d})$, where $d$ is the diameter of the random graph. This demonstrates the following non-monotone behaviour: the query complexity jumps down at moments when the diameter gets larger; yet, between these moments the query complexity grows. We also show that there exists a non-adaptive algorithm that reconstructs the random graph with $O(n^{4-d}p^{2-d}\ln n)$ distance queries with high probability, and this is best possible.
\end{abstract}

\section{Introduction}

Reconstruction of graphs was thoroughly studied in many different contexts 
 and has various applications, e.g., in DNA sequencing~\cite{DNA,DNA2,DNA3}, phylogenetics~\cite{phylo1,KLW,phylo3}, and recovering neural networks~\cite{Soudry}. A vast amount of literature is devoted to the average-case problem, i.e. to the reconstruction of random graphs~\cite{Bol_reconstruction,DPZ,HT-balls,JKRS,LM2017,MZ21,MosselRoss,Muller,SpinozaWest}. In particular, the conjecture of Kelly and Ulam~\cite{Kelly,Ulam} which is considered as one of the main reconstruction challenges is known to be true in binomial random graphs whp\footnote{With high probability, that is, with probability tending to 1 as $n\to\infty$.}~\cite{Bol_reconstruction,LM2017,Muller}.

Let $G$ be a connected graph on $[n]:=\{1,\ldots,n\}$, and let $\mathcal{Q}\subset{[n]\choose 2}$ be a set of pairs of vertices. In the usual way, we denote by $d_G(x,y)$ the graph distance between $x$ and $y$ in $G$. 
 Let us say that $\mathcal{Q}$ {\it reconstructs} $G$, if $G$ is the only graph on $[n]$ with distances $d_G(x,y)$ between $\{x,y\}\in\mathcal{Q}$, i.e. for every graph $G'$ on $[n]$ such that $d_{G'}(x,y)=d_{G}(x,y)$ for all  $\{x,y\}\in\mathcal{Q}$, we have $G'=G$. We will also call the graph {\it $\mathcal{Q}$-reconstructible} in this case.

In this paper, we study the model of reconstruction with an access to distance query oracle introduced in~\cite{BEE06}, although its variants were considered long before that for modelling reconstruction of a phylogenetic tree~\cite{Hein,KZZ03,WSSB}. For every queried pair of vertices, the oracle answers the distance between these vertices in the hidden unknown graph $G$. For a queried pair of vertices $\{x,y\}$, we denote the oracle's answer to this query by $d(x,y)$, omitting the unknown graph in the subscript. An adaptive algorithm, at every step, selects next query (a pair of vertices) based on the responses from the oracle to earlier queries. 

Let $q\in\mathbb{N}$.  If there exists an adaptive algorithm $\mathtt{A}$ such that, for the hidden input graph $G$ on $[n]$, it queries at most $q$ pairs of vertices comprising a set $\mathcal{Q}\subset{[n]\choose 2}$ that reconstructs $G$, then we call $G$ {\it $q$-reconstructible by $\mathtt{A}$}. Let us call the minimum $q$ such that $G$ is $q$-reconstructible by some algorithm {\it the (distance) query complexity of $G$}. It is easy to see that, in the worst case, the query complexity equals ${n\choose 2}$: all pairs should be queried in order to reconstruct the graph. Reyzin and Srivastava~\cite{RS} showed that, even for some trees, $\Omega(n^2)$ pairs are required to query. On the other hand, Mathieu and Zhou~\cite{MZ21} presented an algorithm such that whp a uniformly random $d$-regular graph on $[n]:=\{1,\ldots,n\}$ is $\lfloor n\ln^2n\rfloor$-reconstructible by this algorithm, and wondered whether methods similar to those developed in their paper are applicable to sparse binomial random graphs $G(n,p)$. In this paper, in contrast, we study the distance query complexity of relatively dense random graphs $G(n,p)$ assuming $p>n^{-1+\varepsilon}$ for some $\varepsilon>0$, i.e., when the random graph has bounded diameter whp. Although our main motivation lies in attempting to achieve tight bounds for the average-case complexity in a situation where the general setup is hard to analyse, additional interest in this problem is sparked by a surprising phenomenon, which will be discussed further.

One can expect that  dense graphs are typically harder to recover through distance queries, as the influence of a single edge to be recovered on distances between vertices might be less pronounced and thus harder to detect. In order to get some intuition of how the query complexity of the random graph evolves, let us first assume that $p$ is significantly above the threshold for the property of having diameter at most 2, namely, 
$$
p=\sqrt{(2\ln n+\ln\ln n+\omega(1))/n}.\footnote{Here and later $\omega(1)$ stands for a function tending to infinity with the underlying parameter $n$, perhaps arbitrarily slowly.}
$$
In this case, whp every pair of non-adjacent vertices has at least two common neighbours. Indeed, probability that a fixed pair of vertices has at most one common neighbour equals
$$
 \mathbb{P}(\mathrm{Bin}(n-2,p^2)\leq 1)=(1-p^2)^{n-2}+(n-2)p^2(1-p^2)^{n-3}=O\left(np^2e^{-np^2}\right)=o(n^{-2}).
$$
Therefore, by the union bound, every non-adjacent pair of vertices of $G(n,p)$ has at least two common neighbours whp. Then whp the query complexity equals ${n\choose 2}$. Indeed, for any set $\mathcal{Q}$ of pairs of vertices of size less than ${n\choose 2}$, changing any single adjacency relation of a pair from ${[n]\choose 2}\setminus\mathcal{Q}$ does not influence the distance between $x$ and $y$ for any $\{x,y\}\in\mathcal{Q}$. Therefore, any set $\mathcal{Q}$ that does not contain all pairs does not reconstruct $G(n,p)$ whp.

On the other hand, it seems likely that for $p$ close to the connectivity threshold $\ln n/n$, the query complexity might be much less --- quasilinear, as in the case of random regular graphs. Although in general the query complexity does not decrease as the graph becomes sparser since, in particular, as we mention above, for certain trees, the query complexity equals $\Omega(n^2)$, we show that, on average, as density decreases, the query complexity jumps down at specific moments around hitting times of diameter's increments. However, between these moments the query complexity grows quite rapidly. In particular, there exist $n^{-1+\varepsilon}\ll p_1\ll p_2\ll n^{-1/2}$ such that whp the diameter of $G(n,p_1)$ is larger than the diameter of $G(n,p_2)$ while the query complexity of the former is also larger. 

\paragraph{Related work.}

Kannan, Mathieu, and Zhou~\cite{KMZ,MZ13} presented a reconstruction algorithm that uses $\tilde O(n^{3/2})$ distance queries for bounded degree graphs. They also proved an information--theoretic lower bound $\Omega(n\log n/\log\log n)$ for trees with maximum degree 3 and asked whether  $\tilde O(n)$ is achievable for all bounded degree graphs. Tight results for certain families of bounded degree graphs --- trees, chordal graphs, and graphs with bounded treelength --- were recently obtained by Bastide and Groenland in~\cite{BG23}. In particular, they proved that for every $\Delta\geq 3$, any randomised algorithm requires $\Omega(n\log n)$ queries to reconstruct $n$-vertex trees of maximum degree $\Delta$ for a certain sequence of sizes of trees, which is tight since there exists a deterministic algorithm to reconstruct such trees using $O(n\log n)$ queries.  In~\cite{MZ21}, Mathieu and Zhou answered the question from~\cite{KMZ} for random $d$-regular graphs with constant $d$: whp $n\ln^2n$ distance queries is enough to reconstruct a uniformly random $d$-regular graph.

A modified version of the problem where the hidden graph is a tree and it is only allowed to query pairs of leaves (that are known) comes from biology: it serves as a model of reconstructing evolutionary (phylogenetic) trees and was introduced by Waterman, Smith, Singh, and Beyer~\cite{WSSB}. This problem is well studied: lower bounds on query complexity for deterministic and randomised algorithms were obtained in~\cite{KZZ03} and~\cite{BG23} respectively, and upper bounds were investigated in~\cite{BFPO,KLW,KZZ03,LOO}.

Other types of query oracles explored in the literature include {\it all-shortest-path}, {\it all-distances}, and {\it shortest-path} queries. The first two oracles answer all shortest paths~\cite{BEE06,SM10} and all distances~\cite{BEE06,EHM06} from a queried vertex respectively. The shortest-path oracle answers a shortest path between a queried pair of vertices~\cite{KMZ,RS}. Erlebach, Hall, and Mihal'\'{a}k~\cite{EHM07} studied reconstruction of $G(n,p)$ from all-shortest-path queries 
 for $p=\mathrm{const}$ and $p=n^{-1+\varepsilon}$ for specific $\varepsilon>0$.

Mathieu and Zhou~\cite{MZ21} also showed that their construction of the set of distance queries can be used to get an upper bound for the metric dimension of random $d$-regular graphs. Let us recall that the {\it metric dimension} $\beta(G)$ of a graph $G$ is the minimum cardinality of $S\subset V(G)$ such that all vertices of $G$ have different vectors of distances to the vertices from $S$. Such {\it resolving sets} $S$ or their variants were used, in particular, to label canonically vertices of random graphs~\cite{BES,Kucera}. The metric dimension of $G(n,p)$ was studied by Bollob\'{a}s, Mitsche, and Pra\l at in~\cite{BMP}. In particular, they proved that the metric dimension of $G(n,p)$ undergoes the following `zigzag' behaviour: if $p=n^{-\alpha}$, then $\log_n \beta(G(n,p))=1-(1-\alpha)\lfloor 1/(1-\alpha)\rfloor$ whp; see also~\cite{Odor} for related results. Although in our proof of the upper bound for the distance query complexity of $G(n,p)$ we use a resolving set whose size is far from being optimal, as we show below the `jumps' of the metric dimension and of the query complexity are synchronous --- both happen around $n^{-k/(k+1)}$, $k\in\mathbb{Z}_{>0}$.

\paragraph{Our contribution.}

As follows from the next result, the query complexity drops for the first time when $p$ passes a threshold $n^{-1/2+o(1)}$ on its way down. More generally, we prove that, for every integer constant $k\geq 1$, as soon as $p$ passes a threshold $n^{-(k+1)/(k+2)+o(1)}$, the query complexity drops by a factor of $\Theta(np)$.

\begin{theorem}
Let $k\geq 1$ be a fixed integer, $\varepsilon>0$,
$$
 n^{-1+\varepsilon}< p<n^{-k/(k+1)-\varepsilon},
$$
and let $G_n\sim G(n,p)$. Then there exist $C=C(k)>0$ and an adaptive algorithm $\mathtt{A}$ such that, for 
$q=\left\lfloor C/(n^{k-2}p^k)\right\rfloor,$ whp $G_n$ is $q$-reconstructible by $\mathtt{A}$.
\label{th1}
\end{theorem}

We prove Theorem~\ref{th1} in Section~\ref{sc:proof1}.

\begin{theorem}
Let $k\geq 1$ be a fixed integer, $\varepsilon>0$,  and let
$$
p=\left(\frac{2\ln n+\omega(1)}{n^{k+1}}\right)^{1/(k+2)}.
$$
Then whp $G_n\sim G(n,p)$ has query complexity at least $\frac{1}{2(k+1+\varepsilon)n^{k-2}p^k}$.
\label{th2}
\end{theorem}

Theorem~\ref{th2} is proven in Section~\ref{sc:proof2}. Let us emphasise that its proof is based on an explicit argument which is more efficient than the usual information-theoretic approach: we show a lower bound on the query complexity of a deterministic graph in terms of its diameter and maximum degree, see Claim~\ref{cl:deterministic_lower_bound}. Unlike the information-theoretic approach, our approach allows to bound the query complexity of any given graph. So, it actually demonstrates that, when the number of queries is less than the bound, even knowing the input graph $G$ does not help to reconstruct it --- no matter which pairs of vertices are queried, there is always a graph $G'\neq G$ with exactly the same answers.

Note that, for 
\begin{equation}
\left(\frac{2\ln n+\omega(1)}{n^{k+1}}\right)^{1/(k+2)}=p<n^{-k/(k+1)-\varepsilon},
\label{eq:p_tights_bounds_adaptive}
\end{equation}
the bounds in Theorem~\ref{th1} and Theorem~\ref{th2} differ by a constant factor, see Figure~\ref{fig}. So, roughly speaking (ignoring the constant factor) the query complexity in this interval increases with $p$.\\

 \begin{figure}[!ht] \centering
\includegraphics[width=420pt]{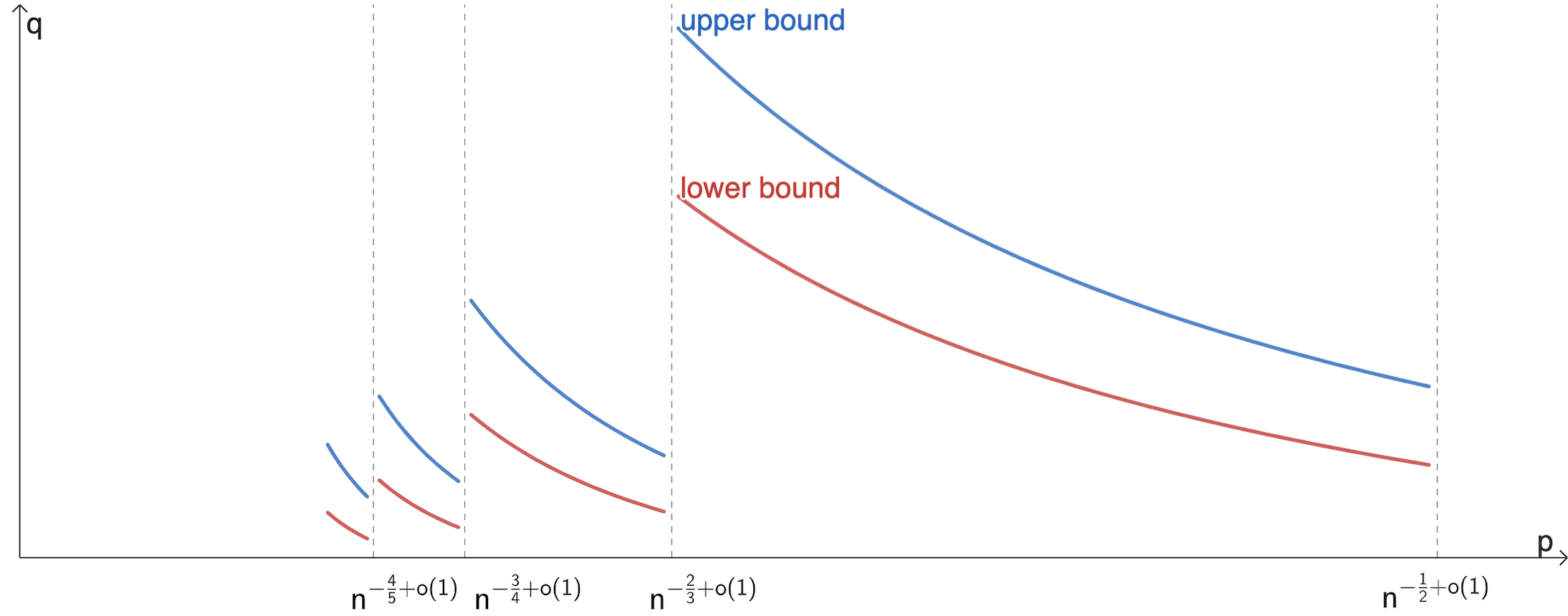}
\caption{\small Upper and lower bounds for query complexity of $G(n,p)$ for $p$ satisfying~\eqref{eq:p_tights_bounds_adaptive}.} \label{fig}
\end{figure}

We also prove tight (up to a constant factor) bounds for ``non-adaptive reconstructibility''. It turns that if the set of queried pairs $\mathcal{Q}$ is fixed, then both upper and lower bounds on the minimum $|\mathcal{Q}|$ such that $G(n,p)$ is $\mathcal{Q}$-reconstructable increase by a factor of $\log n$ when $(2\ln n+\omega(1))^{1/(k+2)} n^{-(k+1)/(k+2)}=p<n^{-k/(k+1)-\varepsilon}$.

\begin{theorem}
Let $k\geq 1$ be a fixed integer, $\varepsilon>0$, and let $p>n^{-1+\varepsilon}$. Then there exists a set $\mathcal{Q}=\mathcal{Q}(n)\subset{[n]\choose 2}$ of size at most $\frac{3}{n^{k-2}p^k}e^{n^kp^{k+1}}\ln n$ such that whp $G_n\sim G(n,p)$ is $\mathcal{Q}$-reconstructible.
\label{th1_2}
\end{theorem}

Theorem~\ref{th1_2} is proven in Section~\ref{sc:proof1_2}.

\begin{theorem}
Let $k\geq 1$ be a fixed integer, $\varepsilon>0$, and let
$$
\left(\frac{2\ln n+\omega(1)}{n^{k+1}}\right)^{1/(k+2)}=
p<\left(\left(\frac{1}{k+1}-\varepsilon\right)\frac{\ln n}{n^k}\right)^{1/(k+1)}.
$$
 Then there is no set $\mathcal{Q}=\mathcal{Q}(n)$ of size at most $\frac{1}{8(k+1)^2 n^{k-2}p^k}e^{n^kp^{k+1}}\ln n $ such that whp $G_n\sim G(n,p)$ is $\mathcal{Q}$-reconstructible.
\label{th3}
\end{theorem}

Theorem~\ref{th3} is proven in Section~\ref{sc:proof3}. Note that there is still an almost polynomial gap between lower and upper bounds even for the ``non-adaptive reconstructibility'' in the narrow ``transition range'', i.e., when
\begin{equation}
(1-o(1))\left(\frac{\ln n}{(k+1)n^k}\right)^{1/(k+1)}<p<\left(\frac{2\ln n+O(1)}{n^k}\right)^{1/(k+1)}.
\label{eq:p_NO_tights_bounds}
\end{equation}

\paragraph{Proofs outline.} For simplicity of presentation, let us assume that $p=n^{-\alpha}$, where $\frac{k}{k+1}<\alpha<\frac{k+1}{k+2}$ for some $k\in\mathbb{Z}_{>0}$.

The crucial observation that allows to prove tight (up to a constant) upper bounds both for adaptive and non-adaptive query complexities is the following: if $\{v_1,x\}$ and $\{v_2,x\}$ are queried and $d(v_2,x)\geq d(v_1,x)+2$, then $v_1$ and $v_2$ are not adjacent. 

Although whp the diameter of $G_n\sim G(n,p)$ equals $k+2$, for a sufficiently large fixed set $X\subset[n]$, every vertex $v\in[n]$ has a vertex from $X$ in its $k$-neighbourhood. Indeed, if $X$ has size $\Theta\left(\frac{\ln n}{n^{k-1}p^k}\right)$, then the expected number of $k$-paths from a fixed vertex $v\in[n]$ to $X$ is $\Theta(\ln n)$. Thus, using standard concentration inequalities, it is possible to overcome the union bound over the choice of $v\in[n]$ for $|X|=C\frac{\ln n}{n^{k-1}p^k}$ and $C$ sufficiently large. Actually, the following stronger property can be proved: whp for every $v_1,v_2\notin X$, there exists $x\in X$ such that $d_{G_n}(v_1,x)\leq k$ while $d_{G_n\setminus\{v_1,v_2\}}(v_2,x)\geq k+2$, where $G_n\setminus\{v_1,v_2\}$ is a spanning subgraph of $G_n$ that is obtained by excluding the edge $\{v_1,v_2\}$. Thus, as soon as all pairs of vertices with at least one vertex in $X$ are queried, all the remaining adjacencies are reconstructible whp. Indeed, for non-adjacent $v_1,v_2$, there exists $x$ such that $d(v_1,x)\leq k$ and $d(v_2,x)\geq k+2$, thus we may reconstruct all non-adjacencies. All the other pairs of vertices must be adjacent since otherwise if we could exclude, say, the edge $\{v_1,v_2\}$, then there is an $x\in X$ such that $d(v_2,x)\leq k+1$ while, in the corrupted graph, the distance between $v_2$ and $x$ would be at least $k+2$. This proves Theorem~\ref{th1_2}, see details in Section~\ref{sc:proof1_2}.

The proof of Theorem~\ref{th1} in Section~\ref{sc:proof1} is based on similar ideas but also utilises the possibility of updating queries adaptively: indeed, instead of querying all pairs $\{x\in X, v\in[n]\}$ simultaneously, we may split the set $X$ into $\nu=\Theta(\log n)$ equal parts $X_1,\ldots,X_{\nu}$ and then, for every $i=1,\ldots,\nu$, query a pair $\{x\in X_i,v\in[n]\}$ only if $v$ has not yet been excluded from the consideration. A vertex $v$ is excluded from the consideration at time $i$, if it has sufficiently many (at least $M$ for some large enough constant $M$) vertices in the intersection of its $k$-neighbourhood with $X_1\cup\ldots\cup X_i$. This allows to query the desired number of pairs whp. In order to complete the proof, it remains to observe that whp there are no $v_1,v_2$ that have $M$ different vertices $x\in X$ so that $d_{G_n\setminus\{v_1,v_2\}}(v_1,x)\leq k$ and $d_{G_n\setminus\{v_1,v_2\}}(v_2,x)\leq k+1$. This would imply the existence of a good $x\in X$ as in the proof of Theorem~\ref{th1_2} for any pair $\{v_1,v_2\}$.

Proofs of both lower bounds in Theorem~\ref{th2} and Theorem~\ref{th3} use the notion of an {\it undetectable pair} in a graph $G$ having diameter $d\geq 3$ with a set of queried pairs of vertices $\mathcal{Q}\subset{V(G)\choose 2}$: a pair of vertices $\{u_1,u_2\}$ is undetectable if $\{u_1,u_2\}\notin\mathcal{Q}\cup E(G)$ and there is no ``path'' $u_1w_1\ldots w_{\ell}u_2$ in $G$, where $\ell\leq d-2$, such that $\ell$ of its ``edges'' are actual edges of $G$, and the one remaining pair of consecutive vertices belongs to $\mathcal{Q}$. It is easy to see that if $G$ has at least one undetectable pair, then it is not reconstructible from $\mathcal{Q}$ since the addition of the edge $\{u_1,u_2\}$ does not change answers to queries. When $\mathcal{Q}$ is small, then a simple counting proof shows that there are undetectable pairs in $G_n$ whp.  This proves Theorem~\ref{th2}, see details in Section~\ref{sc:proof2}.

In order to improve the lower bound by a factor of $\log n$ for the ``non-adaptive reconstructibility'' (Theorem~\ref{th3}), we first refine the notion of undetectability, and then introduce a randomised algorithm that finds an undetectable pair (in the refined sense) in $G_n$ whp. Let us call a pair of vertices $\{u_1,u_2\}$  $\ell$-undetectable in $G$, if $\{u_1,u_2\}\notin\mathcal{Q}\cup E(G)$ and there is no ``path'' $u_1w_1\ldots w_{\ell}u_2$ in $G$ such that $\ell$ of its ``edges'' are actual edges of $G$, the one remaining pair of consecutive vertices belongs to $\mathcal{Q}$, and the distance between the two vertices in this pair is at least $\ell+2$. Though the refined notion is weaker, it still allows to prove non-reconstructibility: if $G$ has diameter $d$ and has a pair which is $\ell$-undetectable for all $\ell\leq d-2$, then $G$ is not $\mathcal{Q}$-reconstructible. Notice that, to decide whether a given pair $\{u_1,u_2\}$ is undetectable, it is sufficient to know only the edges induced by the union of $k$-neighbourhoods of $u_1$ and $u_2$. In Section~\ref{sc:proof3}, we show that the following algorithm finds an undetectable pair whp: at every step, choose a pair of vertices $\{u_1,u_2\}$ uniformly at random from the set of {\it unconsidered} vertices, consider all vertices in $k$-neighbourhoods of $u_1,u_2$, and check whether $\{u_1,u_2\}$ is $\ell$-undetectable for all $\ell\leq k$.

\paragraph{Structure of the paper.} The rest of the paper is organised as follows. In Section~\ref{sc:rg_pre} we recall properties of random graphs that we will use in our proofs: namely, asymptotic behaviour of sizes of balls of small radii and the diameter. Proofs of upper bounds --- Theorem~\ref{th1} and Theorem~\ref{th1_2} --- appear in Sections~\ref{sc:proof1}~and~\ref{sc:proof1_2} respectively. In Sections~\ref{sc:proof2}~and~\ref{sc:proof3} we prove the lower bounds, Theorems~\ref{th2}~and~\ref{th3} respectively. Finally, in Section~\ref{sc:further} we raise several natural questions that we leave unanswered.

\section{Properties of random graphs}
\label{sc:rg_pre}

In this section, we describe certain properties of $G_n\sim G(n,p)$ that we will use in the proofs. We also refer the reader to monographs~\cite{Bollobas,FriezeKaronski,Janson} for a comprehensive exposition of properties of random graphs and of probabilistic tools --- especially, to \cite[Chapter 2]{Janson} containing well-known standard concentration inequalities (in particular, Chernoff bounds for binomial distribution~\cite[Theorem 2.1]{Janson}, Hoeffding's exponential tail bound for the hypergeometric distribution~\cite[Theorem 2.10]{Janson}, and Harris lemma~\cite[Theorem 2.12]{Janson}) that we will use in the proofs and will not state them explicitly.

Let us state a technical claim about the sizes of $k$-neighbourhoods in $G_n$ that we will use very frequently both for lower and upper bounds. For $v\in[n]$ and $r\in\mathbb{Z}_{\geq 0}$, everywhere in the paper we denote by $N_r(v)$ the (random) $r$-ball around $v$ in $G_n$. We will also denote $N_r(U):=\cup_{v\in U}N_r(v)$ for $U\subset[n]$.

\begin{claim}
Let $k$ be a positive integer and let  $\log^2 n/n\ll p\ll (n^{k-1}\log n)^{-1/k}$. Fix $v\in[n]$. With probability $1-o(n^{-2})$, in $G(n,p)$,
$$
|N_r(v)|=(np)^r(1+o(1/\log n))\quad\text{ for all }r\in[k].
$$
\label{cl:balls}
\end{claim}

\begin{proof}
From the Chernoff bound, it follows that $|N_1(v)|=np(1+o(1/\log n))$ with probability $1-o(n^{-2})$. Let us prove by induction over $r$ that, for every, $r\leq k$, $|N_r(v)|=(np)^i(1+o(1/\log n))$ with probability $1-o(n^{-2})$. Assume that for an $r\in[k-1]$, this probability bound is already obtained. Let us expose the $r$-ball around $v$ and denote the set of vertices at distance exactly $r$ by $N'_r$. Due to the induction assumption, we may suppose that $|N'_r|=(np)^r(1+o(1/\log n))=o(n/\log n)$. Thus, the number of vertices in $[n]\setminus N_r(v)$ having neighbours in $N'_r$ (denoted by $|N'_{r+1}|$) is distributed as $\mathrm{Bin}(n(1-o(1/\log n)),1-(1-p)^{|N'_r|})$. Clearly, $N_{r+1}(v)=N_r(v)\sqcup N'_{r+1}$. Since $\mathbb{E}|N'_{r+1}|=n^{r+1}p^{r+1}(1+o(1/\log n))\gg\log^4 n$, we get that
$$
\mathbb{P}\left(\left||N'_{r+1}|
-\mathbb{E}|N'_{r+1}|\right|>\mathbb{E}|N'_{r+1}|/(\ln n)^{3/2}\right)=\exp\left[-\Omega(\mathbb{E}|N'_{r+1}|/\ln^3 n)\right]=o(n^{-2})
$$
by the Chernoff bound, as needed.
\end{proof}

Let us recall the result of Bollob\'{a}s~\cite{Bol_diameter} about the asymptotic distribution of the diameter of $G_n$, see also~\cite[Chapter 10.2]{Bollobas}.

\begin{theorem}
Let $c\in\mathbb{R}$ be a constant and $p=\left(\frac{2\log n+c}{n^{k}}\right)^{1/(k+1)}$. Then the diameter of $G_n$ converges in distribution to a random variable taking values $k+1$ and $k+2$ with probabilities $e^{-e^{-c}/2}$ and $1-e^{-e^{-c}/2}$ respectively.
\end{theorem}

Since the property of having diameter at most $d$ is monotone, we immediately get the following corollary.

\begin{corollary}
If $p=\left(\frac{2\log n+\omega(1)}{n^k}\right)^{1/(k+1)}$, then whp the diameter of $G_n$ is at most $k+1$. If $p=\left(\frac{2\log n-\omega(1)}{n^k}\right)^{1/(k+1)}$, then whp the diameter of $G_n$ is at least $k+2$.
\label{cor:diameter}
\end{corollary}

\section{Proof of Theorem~\ref{th1}}
\label{sc:proof1}

Let us recall that $k\geq 1$ is a fixed integer, $\varepsilon>0$,  and $n^{-1+\varepsilon}< p<n^{-k/(k+1)-\varepsilon}$. We let $G_n\sim G(n,p)$.

Let $M\in\mathbb{Z}$ be large enough and let us choose a large enough real $C\gg M$. Let us also set $\nu:=\lceil\log_2 n\rceil$ and fix a set $X\subset[n]$ of size $\nu\left\lfloor C/(2n^{k-1}p^k)\right\rfloor$. We will need the following auxiliary statements.

\begin{lemma}
Let $i=i(n)\in[\nu]$ and let $X'\subset X$ be a fixed set of size $i\left\lfloor C/(n^{k-1}p^k)\right\rfloor$. Then with probability $1-o(1/\log n)$, at least $(n-|X|)(1-2^{-i})$ vertices of $[n]\setminus X$ have at least $M$ vertices at distance at most $k$ in $X'$.
\label{lm:tight_main}
\end{lemma}

For an edge $e$ of a graph $G$, we denote by $G\setminus e$ the spanning subgraph of $G$ obtained by deleting $e$. If $e\notin E(G)$, then we let $G\setminus e=G$.

\begin{lemma}
Whp there are no $v_1,v_2\in[n]$ and $u_1,\ldots,u_{M}\in X$ such that 
\begin{equation}
d_{G_n\setminus\{v_1,v_2\}}(v_1,u_j)\leq k\quad\text{ and }\quad d_{G_n\setminus\{v_1,v_2\}}(v_2,u_j)\leq k+1
\label{eq:not_dense}
\end{equation}
for all $j\in[M]$.
\label{cl:k_vs_k+2}
\end{lemma}

We prove Lemma~\ref{lm:tight_main} and Lemma~\ref{cl:k_vs_k+2} in Sections~\ref{sc:lm_proof}~and~\ref{sc:cl_proof} respectively. \\

Let us consider a balanced partition of $X$ into $\nu$ disjoint sets $X_1,\ldots,X_{\nu}$ and assume that $G_n$ is a graph satisfying conclusions of Lemma~\ref{lm:tight_main} and Lemma~\ref{cl:k_vs_k+2} deterministically, i.e.:
\begin{itemize}
\item[$\mathbf{P1}$] for every $i\in[\nu]$, the number of vertices $v\in[n]\setminus X$ that have at least $M$ vertices $u\in X_1\cup\ldots\cup X_i$ such that $d(v,u)\leq k$ is at least $(n-|X|)(1-2^{-i})$;
\item[$\mathbf{P2}$] there are no $v_1,v_2\in[n]$ and $u_1,\ldots,u_{M}\in X$ such that~\eqref{eq:not_dense} holds for all $j\in[M]$.
\end{itemize}
Let us prove that these conclusions imply the statements of Theorem~\ref{th1}.\\ 

Let us first query all pairs from ${X\choose 2}$ in $G_n$. Then, we query all pairs from $X_1\times ([n]\setminus X)$ and find all vertices $v\notin X$ such that there are at least $M$ vertices $u\in X_1$ satisfying $d(v,u)\leq k$. Let $V_1\subset[n]\setminus X$ be the set of all such vertices $v$. If $V_1=[n]\setminus X$, then we set $\hat i:=1$. We then proceed by induction: assume that at step $i\in\{1,\ldots,\nu-1\}$ disjoint sets $V_1,\ldots,V_i\subset[n]\setminus X$ are chosen and $V_1\cup\ldots\cup V_i\neq [n]\setminus X$. We then query all pairs from $X_{i+1}\times ([n]\setminus (X\cup V_1\cup\ldots\cup V_i))$ and find all vertices $v\notin X\cup V_1\cup\ldots\cup V_i$ such that there are at least $M$ vertices $u\in X_1\cup\ldots\cup X_i \cup X_{i+1}$ satisfying $d(v,u)\leq k$. We let $V_{i+1}\subset [n]\setminus (X\cup V_1\cup\ldots\cup V_i)$ be the set of all such $v$. If $V_1\cup\ldots\cup V_{i+1}=[n]\setminus X$, then we set $\hat i:=i+1$. From $\mathbf{P1}$, we get that, for every $i$, $|V_1\cup\ldots\cup V_i|\geq (1-2^{-i})(n-|X|)$. Thus, the inductive procedure stops at some moment $\hat i<\nu$, i.e. $V_1\cup\ldots\cup V_{\hat i}=n\setminus X$, and we query 
\begin{align*}
{|X|\choose 2}+\sum_{i=1}^{\hat i}|X_i|\cdot|[n]\setminus (X\cup V_1\cup\ldots\cup V_{i-1})|
 & \leq \left\lfloor C/(2n^{k-1}p^k)\right\rfloor \sum_{i=1}^{\hat i} n\cdot 2^{1-i}\\
 &<2n\left\lfloor C/(2n^{k-1}p^k)\right\rfloor \leq q
\end{align*} 
pairs.

Let us now show that $G_n$ is reconstructible from the set of queried pairs. For every $i\leq\hat i$, we have to recover all adjacencies between $v_1\in V_i$ and $v_2\in [n]$. Without loss of generality we may assume that $v_2$ is either from $X$ or from $V_j$ for some $j>i$. Since all pairs from $V_i\times(X_1\cup\ldots\cup X_i)$ are queried, we may consider only $v_2\in (X_{i+1}\cup\ldots\cup X_{\nu})\cup(V_{i+1}\cup\ldots\cup V_{\hat i})$. It is crucial that for all such $v_2$, the pairs $\{v_2,u\}$ for $u\in X_1\cup\ldots\cup X_i$ are queried, so that we got both $d(v_1,u)$ and $d(v_2,u)$ from the oracle.

Let us first assume that $v_1$ and $v_2$ are non-adjacent in $G_n$. Let $u_1,\ldots,u_{M}\in X_1\cup\ldots\cup X_i$ be at distance at most $k$ from $v_1$. From $\mathbf{P2}$, it follows that there exists $u\in\{u_1,\ldots,u_{M}\}$ such that $d(v_1,u)\leq k$ while $d(v_2,u)\geq k+2$. It may only happen when $v_1,v_2$ are not adjacent. Thus, we may ``reconstruct non-adjacencies'' for all pairs of non-adjacent vertices in $G_n$. 

Now, assume that $G'$ is obtained from $G_n$ by removing some edges that belong to $\cup_{i\leq\hat i} V_i\times ([n]\setminus(X_1\cup\ldots\cup X_i))$ and that for $G'$ we get exactly the same answers as for $G_n$ for all queried pairs. Let $\{v_1,v_2\}\in G_n\setminus G'$, where $v_1\in V_i$ and $v_2\in(X_{i+1}\cup\ldots\cup X_{\nu})\cup(V_{i+1}\cup\ldots\cup V_{\hat i})$. As above, we let $u_1,\ldots,u_{M}\in X_1\cup\ldots\cup X_i$ be at distance at most $k$ from $v_1$ in $G_n$. From $\mathbf{P2}$, we get that there exists $u\in\{u_1,\ldots,u_{M}\}$ such that $d_{G'}(v_2,u)\geq k+2$. This inequality holds true since distances in $G'$ cannot be less than in $G_n$. But then $d_{G_n}(v_2,u)\geq k+2$ as $d_{G_n}(v_2,u)=d_{G'}(v_2,u)$ by the assumption since the pair $\{v_2,u\}$ was queried. On the other hand, $v_1\sim v_2$ in $G_n$ and $d_{G_n}(v_1,u)\leq k$. So $d(v_2,u)\leq d(v_1,u)+1\leq k+1$ --- a contradiction. Thus, edges are also reconstructible, Theorem~\ref{th1} follows.

\subsection{Proof of Lemma~\ref{lm:tight_main}}
\label{sc:lm_proof}

Fix $i\in[\nu]$ and $X'\subset X$ of size $i\left\lfloor C/(2n^{k-1}p^k)\right\rfloor$. For $v\in[n]\setminus X$, we let $\mathcal{B}_v$ be the event saying that at least $M$ vertices in $X'$ are at distance at most $k$ from $v$. We denote $\xi:=\sum_{v\in[n]\setminus X}\1_{\mathcal{B}_v}$ the number of such vertices. We have to prove that 
$$
\mathbb{P}(\xi<(n-|X|)(1-2^{-i}))=o(1/\log n).
$$ 

Fix a vertex $v\in[n]\setminus X$. For any $N\in[n]$, $N_k(v)$ is uniformly random $N$-subset of $[n]$ that contains $v$ subject to the event $\{|N_k(v)|=N\}$. We derive the following inequality from Hoeffding's exponential tail bound for the hypergeometric distribution (see~\cite[Theorem 2.10]{Janson}): for every $N=(np)^k(1+o(1))$,
\begin{equation}
 \mathbb{P}\left(|N_k(v)\cap X'|< \frac{1}{4}iC
 \mid |N_k(v)|=N\right)\leq\exp\left[-\frac{iC}{16-o(1)}\right].
 \label{eq:k-ball}
\end{equation}
First of all,~\eqref{eq:k-ball} implies that $\mathbb{P}\left(|N_k(v)\cap X'|< M \mid |N_k|=N\right)=o(1/(n\log n))$ uniformly over $N=(np)^k(1+o(1))$ if $i\geq 17\ln n/C$. By the union bound and Claim~\ref{cl:balls}, we get that $\xi=n-|X|$ for such $i$ with probability $1-o(1/\log n)$. We then further assume that $i<17\ln n/C$ and that $C$ is so large that $e^{-2i}>n^{-2/(k+1)}$.

For $v\in[n]\setminus X'$, denote by $N^{\neg X'}_r(v)$ the $r$-ball around $v$ in $G_n[[n]\setminus X']$. Due to Claim~\ref{cl:balls}, with probability $1-o(1/n)$, $|N^{\neg X'}_{k-1}(v)|=(np)^{k-1}(1+o(1))$ for all $v\in[n]\setminus X'$. We then expose only edges induced by $[n]\setminus X'$ and assume that the latter event holds deterministically.

For $v\in[n]\setminus X$, we let $\mathcal{C}_v$ be the event saying that less than $M$ vertices in $X'$ have neighbours in $N^{\neg X'}_{k-1}(v)$. Let $\eta=\sum_{v\in[n]\setminus X}\1_{\mathcal{C}_v}$. Note that $\eta\geq n-|X|-\xi$. Therefore, it suffices to prove that
$$
\mathbb{P}(\eta\geq 2^{-i}(n-|X|))=o(1/\log n).
$$
A vertex from $X'$ has a neighbour in $N^{\neg X'}_{k-1}(v)$ for a fixed  $v\in[n]\setminus X$ with probability 
$$
1-(1-p)^{(np)^{k-1}(1+o(1))}=1-e^{-n^{k-1}p^k(1+o(1))}\sim n^{k-1}p^k(1+o(1))
$$ 
since $p<n^{-k/(k+1)-\varepsilon}$. Therefore, by the Cherhoff bound,
$$
 \mathbb{P}(\mathcal{C}_v)\leq \mathbb{P}\left(\mathrm{Bin}(|X'|,n^{k-1}p^k(1+o(1))<\frac{1}{4}iC\right)\leq\exp\left[-\frac{iC}{16-o(1)}\right].
$$
We get that $\mathbb{E}\eta\leq e^{-i}(n-|X|)$ for $C>16$ and for large enough $n$.

Since the relation $u\in N^{\neg X'}_{k-1}(v)$ on pairs $(u,v)$ of vertices from $[n]\setminus X'$ is symmetric, every vertex from $[n]\setminus X'$ belongs to at most $(np)^{k-1}(1+o(1))$ balls $N^{\neg X'}_{k-1}(v)$, $v\in[n]\setminus X$. We shall use the following bounded difference inequality applied to $\eta$ considered as a function of edges between $X'$ and $[n]\setminus X'$:
\begin{theorem}[\cite{McDiarmid}]
Let $\xi_1,\ldots,\xi_N\sim\mathrm{Bern}(p)$ be independent, and let $f:\mathbb{R}^N\to\mathbb{R}$ satisfy 
$$
|f(x_1,\ldots,x_{i-1},x_i,x_{i+1},\ldots,x_N)-f(x_1,\ldots,x_{i-1},x_i',x_{i+1},\ldots,x_N)|\leq c_i
$$ 
for some $c_1,\ldots,c_N>0$ and every $i\in[N]$, $x_1,\ldots,x_{N},x_i'\in\mathbb{R}$. Then, for every $t>0$,
$$
 \mathbb{P}(|f(\xi_1,\ldots,\xi_N)-\mathbb{E}f(\xi_1,\ldots,\xi_N)|\geq t)\leq\exp\left[-\frac{t^2}{2p(1-p)\sum_{i=1}^Nc_i^2+\frac{2}{3}t\max_i c_i}\right].
$$
\label{th:McDiarmid}
\end{theorem}

Assuming that $\hat G_n$ is a graph on $[n]$ such that $\hat G_n[n\setminus X']=G_n[n\setminus X']$ contains exactly the exposed edges, we get that, for any pair $e\in X'\times(n\setminus X')$, values $\eta(\hat G_n)$ and $\eta(\hat G_n\bigtriangleup e)$ differ by at most $(np)^{k-1}(1+o(1))$ since $e$ may change only the values of $\1_{\mathcal{C}_v}$ such that $(e\setminus X')\subset N^{\neg X'}_{k-1}(v)$. Therefore, we may set all $c_i$ equal to $(np)^{k-1}(1+o(1))$. Then Theorem~\ref{th:McDiarmid} implies that, for some constant $c>0$,
\begin{align*}
 \mathbb{P}(\eta\geq 2^{-i}(n-|X|)) & \leq \mathbb{P}\left(\eta\geq \frac{e}{2}\cdot e^{-i}(n-|X|)\right)\\
 &\leq\exp\left[-c\frac{e^{-2i}(n-|X|)^2}{p(n-|X'|)|X'|(np)^{2k-2}+e^{-i}(n-|X|)(np)^{k-1}}\right]\\
 & \leq\exp\left[-\frac{2c}{C+1}\cdot\frac{e^{-2i} n}{i(np)^{k-1}}\right]
 \leq \exp\left[-n^{\varepsilon(k-1)+o(1)}\right]=o\left(\frac{1}{\log n}\right),
\end{align*}
where the penultimate inequality is due to $e^{-2i} > n^{-2/(k+1)}$,  completing the proof.

\subsection{Proof of Lemma~\ref{cl:k_vs_k+2}}
\label{sc:cl_proof}

Assume that $M\gg M'\gg 1/\varepsilon.$

\begin{claim}
Fix a vertex $v\in[n]$. With probability $1-o(n^{-2})$, $|N_{k-1}(v)\cap X|\leq M'$. 
\end{claim}

\begin{proof}
Due to Claim~\ref{cl:balls}, we have that $|N_{k-1}(v)|=(np)^{k-1}(1+o(1))$ with probability $1-o(n^{-2})$. For simplicity of presentation, let us assume that $v\notin X$ (the case $v\in X$ does not differ much). Subject to $|N_{k-1}(v)|=N+1$, the random variable $|N_{k-1}(v)\cap X|$ has a hypergeometric distribution, implying that
\begin{align*}
\mathbb{P}(|N_{k-1}(v)\cap X|>M') & \leq(1+o(1))\max_{N=(np)^{k-1}(1+o(1))}\frac{{|X|\choose M'}{n-|X|-1\choose N-M'}}{{n-1\choose N}}\\
&\leq(1+o(1))\left(\frac{|X|-M'}{n}\right)^{M'}=o(n^{-2}).
\end{align*}
\end{proof}

Let us define the following family of graphs $\mathcal{H}$ on $[n]$: fix a vertex $v_1$ and draw from this vertex $\lfloor M/M'\rfloor$ disjoint paths of length at most $k$ to distinct vertices from $X$ and avoiding the edge $\{v_1,v_2\}$; fix another vertex $v_2$ and draw from this vertex a path of length at most $k+1$ to each of the paths starting at $v_1$ and avoiding the edge $\{v_1,v_2\}$: this $\lfloor M/M'\rfloor$ paths from $v_2$ are not necessarily disjoint.

\begin{claim}
Whp $G_n$ does not contain any subgraph from $\mathcal{H}$.
\end{claim}

\begin{proof}
First of all note that any $H\in\mathcal{H}$ has at least $\lfloor M/M'\rfloor$ and at most $2k\lfloor M/M'\rfloor+2$ vertices. Moreover, for any $x$ in this range, the number of graphs in $\mathcal{H}$ with exactly $x$ vertices is at most $c_xn^{x-\lfloor M/M'\rfloor}|X|^{\lfloor M/M'\rfloor}$ for a certain constant $c_x$. Let us bound from below the number of edges in such a graph and apply the union bound. Let us fix vertices $v_1,v_2$ and $u_1,\ldots,u_{\lfloor M/M'\rfloor}$ that play the role of ends of paths from $v_1$ in $X$. Note that a graph from $\mathcal{H}$ can be obtained from this set of vertices by a sequential addition of 2$\lfloor M/M'\rfloor$ paths joining two vertices that are presented in the graph at the previous step. First $\lfloor M/M'\rfloor$ paths are of length at most $k$, and the last $\lfloor M/M'\rfloor$ paths are of length at most $k+1$. Assuming that at the $j$-th step the number of added vertices equals $x_j$, we get that the constructed graph $H$ has density
\begin{align*}
 \rho(H)=\frac{|E(H)|}{|V(H)|} & =\frac{\sum_{j=1}^{2\lfloor M/M'\rfloor}(x_i+1)}{2+\lfloor M/M'\rfloor+\sum_{j=1}^{2\lfloor M/M'\rfloor}x_i}\\
 &=\frac{x-2+\lfloor M/M'\rfloor}{x}\geq\frac{(2k+1)\lfloor M/M'\rfloor}{2k\lfloor M/M'\rfloor+2}.
\end{align*}
Finally, assuming $\lfloor M/M'\rfloor>\frac{3}{\varepsilon(k+1)}$, the expected number of subgraphs $H\in\mathcal{H}$ in $G_n$ is at most
$$
 \sum_x c_xn^{x-\lfloor M/M'\rfloor}|X|^{\lfloor M/M'\rfloor} p^{x\frac{(2k+1)\lfloor M/M'\rfloor}{2k\lfloor M/M'\rfloor+2}}=O\left(\left(\nu n^{k+2/\lfloor M/M'\rfloor}p^{k+1}\right)^{\lfloor M/M'\rfloor}\right)=O(n^{-1}).
$$
\end{proof}

Let us now fix $v_1,v_2$ and assume that the graph $G':=G_n\setminus\{v_1,v_2\}$ satisfies the following properties deterministically:
\begin{itemize}
\item for all $v\in[n]$, $|N_{k-1}(v)\cap X|\leq M'$;
\item $u_1,\ldots,u_{M}\in N_k(v_1)$;
\item $G'$ does not contain a graph from family $\mathcal{H}$.
\end{itemize}
It is then sufficient to prove that there exists $j\in[M]$ such that $u_j\notin N_{k+1}(v_2)$.

Take a shortest path $P_1$ (of length at most $k$) between $v_1$ and $u_1$. 
 We then proceed by induction. Assume that at step $j<\lfloor M/M'\rfloor$, we found $j$ disjoint shortest paths $P_1,\ldots,P_j$ from $v_1$ to $u_1,\ldots,u_j$ respectively. Let $w_1,\ldots,w_j$ be the first vertices (after $v_1$) on $P_1,\ldots,P_j$ respectively. We have that there exists $j'\in[M]\setminus[j]$ such that $u_{j'}\notin N_{k-1}(w_1)\cup\ldots\cup N_{k-1}(w_j)$. Without loss of generality let $j'=j+1$. Observe that the shortest path $P_{j+1}$ between $v_1$ and $u_{j+1}$ cannot intersect any of $P_1,\ldots,P_j$ outside of $v_1$.
 
Eventually we get $\lfloor M/M'\rfloor$ disjoint paths from $v_1$ to $u_1,\ldots,u_{\lfloor M/M'\rfloor}$. If, for some $j\in[\lfloor M/M'\rfloor]$, $u_j\notin N_{k+1}(v_2)$, we are done. Otherwise, $G_n$ contains a graph from $\mathcal{H}$ --- a contradiction.

\section{Proof of Theorem~\ref{th1_2}}
\label{sc:proof1_2}

Let $\lambda=n^kp^{k+1}$. Without loss of generality we may assume that 
\begin{equation}
p\leq \left(\frac{(k+o(1))\ln n}{(k+1)n^k}\right)^{1/(k+1)}
\label{eq:th1_2_p_upper_bound}
\end{equation} 
so that $\frac{3e^{\lambda}\ln n}{n^{k-2}p^k}\leq{n\choose 2}$.
 Let $X$ be a fixed set of vertices of size $\left\lfloor 3e^{\lambda}\ln n/(n^{k-1}p^k)\right\rfloor\leq\frac{n}{2}$. 
  Let us query all the pairs from $[X\times ([n]\setminus X)]\cup{X\choose 2}$. Thus, we query at most $\frac{3}{n^{k-2}p^k}e^{\lambda}\ln n$ pairs. 
As in the previous section, we let $G_n\sim G(n,p)$, and, for $v\in[n]$, we let $N_k(v)$ denote the set of vertices in $G_n$ at distance at most $k$ from $v$.

 We further show that whp in $G_n$ for every $v_1,v_2\in [n]\setminus X$, there is an $x\in X$ such that $x$ is at distance at most $k$ from $v_1$ and at least $k+2$ from $v_2$ in the graph obtained from $G_n$ by removing the edge $\{v_1,v_2\}$. 
  It would immediately imply that whp $G_n$ is $q$-reconstructible. Indeed, we need only to reconstruct adjacency relations between pairs of vertices $v_1,v_2\notin X$. If $v_1,v_2$ are non-adjacent, then we queried $\{v_1,x\}$, $\{v_2,x\}$ such that 
   $d(v_1,x)\leq k$, $d(v_2,x)\geq k+2$. This may only happen when $v_1,v_2$ are non-adjacent. Then we can reconstruct all non-adjacencies. It remains to prove that we cannot delete any subset of edges from $G_n[[n]\setminus X]$ without changing the answers to the queries. If $v_1\sim v_2$, then we queried $\{v_1,x\}$, $\{v_2,x\}$ such that 
    $d(v_1,x)\leq k$, $d(v_2,x)\leq d(v_1,x)+1$, and any shortest path from $x$ to $v_2$ follows the edge $\{v_1,v_2\}$. Thus, the edge $\{v_1,v_2\}$ must be in the graph.
    
Therefore, the following lemma immediately imples Theorem~\ref{th1}. 
 
\begin{lemma}
Whp for every $v_1,v_2\in [n]\setminus X$, there is an $x\in X$ such that 
\begin{center}
$d_{G_n\setminus\{v_1,v_2\}}(v_1,x)\leq k\quad$ and $\quad d_{G_n\setminus\{v_1,v_2\}}(v_2,x)\geq k+2$.
\end{center}
\label{lm:upper_main}
\end{lemma}

\begin{proof} 
Fix vertices $v_1,v_2\notin X$ and remove the edge $\{v_1,v_2\}$ if it is in $G_n$. Let $N'$ be the (closed) $(k-1)$-neghbourhood of $v_1$ in $G_n\setminus[X\cup\{v_2\}]$, and let $N''_i$ be the (closed) $i$-neghbourhood of $v_2$ in $G_n\setminus[N'\cup X]$ for $i\in\{k-1,k\}$. Note that, as soon as $N'$ is exposed, $N''_{k-1}$ and $N''_k$ are defined solely by adjacencies that are entirely outside $N'$ and their cardinalities do not depend on edges induced by $N'$ --- only on the size of $N'$. Therefore, due to Claim~\ref{cl:balls}, we may assume that $|N'|=(np)^{k-1}(1+o(1/\log n))$, $|N''_i|=(np)^i(1+o(1/\log n))$, $i\in\{k-1,k\}$. By the Chernoff bound and due to~\eqref{eq:th1_2_p_upper_bound}, the number of neighbours of vertices from $N''_{k-1}$ in $N'$ is $o((np)^{k-1}/\log n)$ with probability $1-o(n^{-2})$.
 Note that the above makes sense only when $k\geq 2$. In the case $k=1$, we do not need to apply the Chernoff bound since $N'=\{v_1\}$, $N''_{k-1}=\{v_2\}$, and there is no edge $\{v_1,v_2\}$ in the considered graph.

We then exclude from $N'$ the neighbours of $N''_{k-1}$  and get that the refined set $N'_0$ still has the size $(np)^{k-1}(1+o(1/\log n))$. If we can prove that with probability $1-o(n^{-2})$, there exists a vertex $x\in X$ that has a neighbour in $N'_0$ and does not have neighbours in $N''_k\sqcup(N'\setminus N'_0)$, we immediately get the statement of Lemma~\ref{lm:upper_main}:  if $P$ is a path of length at most $k+1$ from $v_2$ to $x$ in $G_n\setminus\{v_1,v_2\}$, then by construction the neighbour of $x$ on $P$ is in $N'_0$. Then the farthest from $x$ vertex in $P\cap N'_0$ has a neighbor in $P\cap N''_{k-1}$ --- a contradiction. The probability bound on the latter event is also immediate since
\begin{multline*}
 \mathbb{P}\biggl(\mathrm{Bin}\left(|X|,\left(1-(1-p)^{|N'_0|}\right)(1-p)^{|N''_k|+|N'\setminus N'_0|}\right)=0\biggr)\\
 =\left(1-\left(1-(1-p)^{(np)^{k-1}(1+o(1/\log n))}\right)(1-p)^{(np)^{k}(1+o(1/\log n))}\right)^{|X|}\\
 \stackrel{\eqref{eq:th1_2_p_upper_bound}}=\exp\left[-(1+o(1))|X|n^{k-1}p^k e^{-\lambda}\right]=o(n^{-2}).
\end{multline*}
\end{proof}

\section{Proof of Theorem~\ref{th2}}
\label{sc:proof2}

Without loss of generality, we assume that $p=o(1)$. Let us show that Theorem~\ref{th2} follows from 

\begin{claim}
Let $d\geq 3$ and $\varepsilon>0$. There exists $C=C(d,\varepsilon)>0$ such that, for every graph $G$ on $[n]$ with diameter at most $d$ and maximum degree $\Delta>C$, its query complexity $q$ satisfies:
\begin{equation}
  q(d-1+\varepsilon)\Delta^{d-2}\geq {n\choose 2}-|E(G)|.
  \label{eq:q}
\end{equation}
\label{cl:deterministic_lower_bound}
\end{claim}

Indeed, whp $G_n$ has $(1+o(1)){n\choose 2}p$ edges, diameter at most $k+2$ (by Corollary~\ref{cor:diameter}), and maximum degree $np(1+o(1))$ (by Claim~\ref{cl:balls}, see also~\cite{Bol:max_degree}), implying that the query complexity of $G(n,p)$ is at least $\frac{1}{2(k+1+\varepsilon)n^{k-2}p^k}$ whp, as needed.

\begin{proof}
Let $q$ be an integer that does not satisfy~\eqref{eq:q}. Let $\mathcal{Q}\subset{[n]\choose 2}$ have size $q$.  Let us call a pair of vertices $\{u_1,u_2\}$ {\it undetectable in $G$} if:
\begin{itemize}
\item $u_1,u_2$ are not adjacent in $G$, 
\item the pair $\{u_1,u_2\}$ was not queried,
\item there is no ``path'' $u_1w_1\ldots w_{\ell}u_2$, $\ell\leq d-2$, such that $\ell$ of its ``edges'' are actual edges of $G$, and the one remaining pair of consecutive vertices belongs to $\mathcal{Q}$.
\end{itemize}

Assuming that the pair $\{u_1,u_2\}$ is undetectable in $G$, observe that the addition of the edge $\{u_1,u_2\}$ agrees with the answers to all the queries. Indeed, if it affects some queried pair $\{w,w'\}$, then there should be a path $w\ldots u_1u_2\ldots w'$ of length at most $d-1$ in $G\cup\{u_1,u_2\}$ --- a contradiction.

It remains to prove that $G$ has an undetectable pair. The number of pairs that do not satisfy the third property of an undetectable pair is at most 
$$
\sum_{\ell=1}^{d-2}(\ell+1)q\Delta^{\ell}<q((d-1)\Delta^{d-2}+(d-2)^2\Delta^{d-3})<q(d-1+\varepsilon)\Delta^{d-2}-q
$$
for an appropriate choice of $C$. On the other hand, the number of pairs that satisfy the first two properties is at least
$$
 {n\choose 2}-|E(G)|-q>q(d-1+\varepsilon)\Delta^{d-2}-q,
$$
implying the existence of an undetectable pair, as needed.
\end{proof}

\section{Proof of Theorem~\ref{th3}}
\label{sc:proof3}

Let 
 $G_n\sim G(n,p)$, $N:=\left\lfloor\frac{e^{\lambda}\ln n}{8(k+1)^2 n^{k-2}p^k}\right\rfloor$, where $\lambda=n^kp^{k+1}$. We assume that a set $\mathcal{Q}\subset{[n]\choose 2}$ of $N$ pairs of vertices from $[n]$ is fixed in advance to be queried. 
 We call a pair $\{u_1,u_2\}$ {\it small} if
\begin{itemize}
\item the number of $\{v_1,v_2\}\in\mathcal{Q}$ such that $v_1\in N_{k'}(u_1)$ and $v_2\in N_{k-k'}(u_2)$ for some $k'\in\{0,1,\ldots,k\}$ is at most 
$$
M:=\frac{7(k+1)(np)^k N}{n^2};
$$
\item there are no $\{v_1,v_2\}\in\mathcal{Q}$ such that $v_1\in N_{k'}(u_1)$ and $v_2\in N_{k-1-k'}(u_2)$ for some $k'\in\{0,1,\ldots,k-1\}$;
\item there are no $\{v_1,v_2\}\in\mathcal{Q}$ such that, for some $k'\in\{0,1,\ldots,k\}$, $v_1\in N_{k'}(u_1)$, $v_2\in N_{k-k'}(u_2)$, and either $v_2\in N_k(u_1)$ or $v_1\in N_k(u_2)$.
\end{itemize}

\begin{claim}
Whp at least $\frac{1}{2}{n\choose 2}$ pairs of vertices are small.
\label{cl:half_pairs_small}
\end{claim}

\begin{proof}
Due to Claim~\ref{cl:balls}, whp, for every $k'\in\{0,1,\ldots,k\}$, for every $\{v_1,v_2\}\in\mathcal{Q}$, there are at most $(1+o(1))(np)^k$ {\it ordered} pairs $(u_1,u_2)$ such that $v_1\in N_{k'}(u_1)$ and $v_2\in N_{k-k'}(u_2)$. Then whp the event that the number of pairs $\{u_1,u_2\}$ satisfying the first condition in the definition of a small pair is less than $\frac{5}{6}{n\choose 2}$ implies 
$$
\frac{1}{6}n(n-1)M<(1+o(1))(k+1)(np)^k N
$$
which is false for large enough $n$.

Similarly, whp, for every $\{v_1,v_2\}\in\mathcal{Q}$, there are at most $(1+o(1))k(np)^{k-1}$ {\it ordered} pairs $(u_1,u_2)$ such that $v_1\in N_{k'}(u_1)$ and $v_2\in N_{k-1-k'}(u_2)$ for some $k'\in\{0,1,\ldots,k-1\}$.
 Then whp the event that the number of pairs $\{u_1,u_2\}$ satisfying the second condition in the definition of a small pair is less than $\frac{5}{6}{n\choose 2}$ implies 
$$
 \frac{1}{6}n(n-1)<(1+o(1))k(np)^{k-1} N
$$
which is false for large enough $n$ as well.

Let us finally show that whp the number of pairs $\{u_1,u_2\}$ satisfying the third condition in the definition of a small pair is at least $\frac{5}{6}{n\choose 2}$. We will do it in two steps: first, we will show that whp there are only $o(n^2)$ pairs $\{u_1,u_2\}$ that have in their ``neighbourhood'' $N_{k'}(u_1)\times N_{k-k'}(u_2)$ queried pairs $\{v_1,v_2\}$ with $d(v_1,v_2)\leq k$; second, we will observe that whp as soon as a queried pair $\{v_1,v_2\}$ has $d(v_1,v_2)>k$, then the intersection of the $k'$-ball around $v_1$ and the $k$-ball around $v_2$ has size around $(np)^{k+k'}/n$ --- that will be sufficient to finish the proof.

Fix a queried pair $\{v_1,v_2\}$. 
 Since the expected number of paths of length at most $k$ between $v_1,v_2$ is at most $n^{k-1}p^k(1+o(1))$, we get that the distance between $v_1,v_2$ is at most $k$ with probability at most $n^{k-1}p^k(1+o(1))$ by Markov's inequality. Therefore, the expected number of pairs in $\mathcal{Q}$ that are at distance at most $k$ is at most $n^{k-1}p^k(1+o(1))\cdot N$. By Markov's inequality, we get that the number of such pairs in $\mathcal{Q}$ is $O_p(n^{k-1}p^kN)$. As we noted above, whp, for every pair $\{v_1,v_2\}\in\mathcal{Q}$, there are at most $(1+o(1))(k+1)(np)^k$ pairs $(u_1,u_2)$ such that $v_1\in N_{k'}(u_1)$, $v_2\in N_{k-k'}(u_2)$ for some $k'\in\{0,1,\ldots,k\}$. 
 Let us fix a sequence $w_n$ growing to infinity with $n$ sufficiently slowly. We get that the event that the number of pairs $\{u_1,u_2\}$ that have $v_1\in N_{k'}(u_1)$ and $v_2\in N_{k-k'}(u_2)$ for some $k'\in\{0,1,\ldots,k\}$ such that $\{v_1,v_2\}\in\mathcal{Q}$ and $d(v_1,v_2)\leq k$ is at least $\frac{1}{12}{n\choose 2}$ implies 
 $$
 \frac{1}{12}n(n-1)<w_n(np)^k n^{k-1}p^k N
 $$
 which is false for all large enough $n$. This proves that whp the number of pairs satisfying the last property in the definition of small pairs with the restriction that $d(v_1,v_2)\leq k$ (note that it always holds for $k'\in\{0,k\})$ is at least $\frac{11}{12}{n\choose 2}$. It then remains to prove that whp the number of pairs $\{u_1,u_2\}$ with no $\{v_1,v_2\}\in\mathcal{Q}$ such that $d(v_1,v_2)>k$ and, for some $k'\in[k-1]$, $v_1\in N_{k'}(u_1)$, $v_2\in N_{k-k'}(u_2)$, and either $v_2\in N_k(u_1)$ or $v_1\in N_k(u_2)$, is at least $\frac{11}{12}{n\choose 2}$ as well.

As promised, we first show that whp, for any $\{v_1,v_2\}$ such that $d(v_1,v_2)>k$ and any $k'\in[k]$, the intersection of the $k'$-ball around $v_1$ and the $k$-ball around $v_2$ has size at most $(1+o(1))(np)^{k+k'}/n+\ln^2 n$. Fix $v_2$ and expose the $k$-ball around it. Due to Claim~\ref{cl:balls}, with probability $1-o(n^{-2})$, this ball $N_k(v_2)$ has size $(1+o(1))(np)^{k}$. Let us now observe that we may assume that all vertices from $S:=N_k(v_2)\setminus N_{k-1}(v_2)$ have degrees at most $\ln^2 n$ in $N_k(v_2)$. Indeed, the complementary event implies that there exists a vertex $w\in S$ such that the total number of $(k+1)$-paths and $k$-paths between $u$ in $v_2$ is at least $\ln^2 n$, which is false with probability $1-o(n^{-2})$: due to~\cite[Theorem 5]{Spencer}, there exists $C>0$ such that, with probability $1-o(n^{-2})$, the number of paths of length at most $k+1$ between any two vertices in $G(n,p)$ is at most $C\ln n$. 

Then, fix $v_1\notin N_k(v_2)$. 
 For $i\in\{0,1,\ldots,k'-1\}$, let $N'_i$ be the $i$-neighbourhood of $v_1$ in $[n]\setminus N_k(v_2)$, and let $N^{\cap}_i$ be the set of neighbours of $N'_i$ in $N_k(v_2)$. Note that, for every vertex $w\in N_{k'}(v_1)\cap N_k(v_2)$, a shortest path from $w$ to $v_1$ leaves $N_k(v_2)$ at some vertex $w'\in S$. Thus, the set $N_{k'}(v_1)\cap N_k(v_2)$ coincides with $\cup_{i=0}^{k'-1}[N_{k'-1-i}(N^{\cap}_i)\cap N_k(v_2)]$. Then
\begin{equation}
|N_k(v_2)\cap N_{k'}(v_1)|\leq\sum_{i=0}^{k'-2}|N^{\cap}_i|\ln^2n (np)^{k'-2-i}+|N^{\cap}_{k'-1}|
\label{eq:intersection_neighbourhoods_bound_Spencer}
\end{equation}
with probability $1-o(n^{-2})$. Moreover, due to Claim~\ref{cl:balls}, with probability $1-o(n^{-2})$, for every $i\in\{0,1,\ldots,k'-1\}$, $|N'_i|\leq(np)^i(1+o(1))$. 

Expose all $N'_i$, $i\in\{0,1,\ldots,k'-1\}$, and assume that $|N'_i|\leq(np)^i(1+o(1))$ and $|N_k(v_2)|\leq(1+o(1))(np)^k$ hold deterministically. Since $|N^{\cap}_i|$ is the number of vertices  in $N_k(v_2)$ that have at least one neighbour in $N'_i$, we get, by the Chernoff bound,
\begin{multline*}
 \mathbb{P}(|N^{\cap}_i|>(np)^{k+i}p(1+(\ln\ln n)^{-1})+\ln^2 n)\leq\\
 \mathbb{P}\biggl(\mathrm{Bin}\left((np)^k(1+o(1)),1-(1-p)^{(np)^i(1+o(1))}\right)>(np)^{k+i}p(1+(\ln\ln n)^{-1})+\ln^2 n\biggr)\\
 =o(n^{-2}).
\end{multline*}
We finally get from~\eqref{eq:intersection_neighbourhoods_bound_Spencer} that with probability $1-o(n^{-2})$,
\begin{multline*}
 |N_{k'}(v_1)\cap N_k(v_2)|\leq \\
 \sum_{i=0}^{k'-2}\ln^2n (np)^{k'-2-i}((np)^{k+i}p(1+o(1))+\ln^2 n)+(np)^{k+k'-1}p(1+o(1))+\ln^2n\\
 =(1+o(1))(np)^{k+k'-1}p+\ln^2 n,
\end{multline*}
as needed.

From this and since whp, for every vertex $v_2$ and every $k'\in[k-1]$, $|N_{k-k'}(v_2)|=(1+o(1))(np)^{k-k'}$ due to Claim~\ref{cl:balls}, we get that whp, for every $\{v_1,v_2\}\in\mathcal{Q}$ satisfying $d(v_1,v_2)>k$, there at most $k\left(\frac{(np)^{2k}}{n}+(np)^{k-1}\ln^2 n\right)$ pairs $(u_1,u_2)$ such that, for some $k'\in[k-1]$,
\begin{equation}
v_1\in N_{k'}(u_1), \quad v_2\in N_{k-k'}(u_2),\quad \text{and}\quad v_2\in N_k(u_1).
\label{eq:screw_pair}
\end{equation}
So, the event that the number of pairs $\{u_1,u_2\}$ that have a pair $\{v_1,v_2\}$ satisfying $d(v_1,v_2)>k$ and~\eqref{eq:screw_pair} is at least $\frac{1}{24}{n\choose 2}$ implies
$$
\frac{1}{24}n(n-1)<k\left(\frac{(np)^{2k}}{n}+(np)^{k-1}\ln^2 n\right)N,
$$
which is false for all large $n$. Due to symmetry, we get that whp the number of pairs $\{u_1,u_2\}$ that have a pair $\{v_1,v_2\}$ satisfying $d(v_1,v_2)>k$, $v_1\in N_{k'}(u_1)$, $v_2\in N_{k-k'}(u_2)$, and  $v_1\in N_k(u_2)$ is less than $\frac{1}{24}{n\choose 2}$ as well, completing the proof.

\end{proof}

Let $\ell\in[k]$. Let us call a pair of vertices $(u_1,u_2)$ {\it $\ell$-undetectable}, if the following three conditions hold:
\begin{itemize}
\item $u_1,u_2$ are not adjacent in $G_n$, 
\item $\{u_1,u_2\}\notin\mathcal{Q}$,
\item there is no ``path'' $u_1w_1\ldots w_{\ell}u_2$, such that $\ell$ of its ``edges'' are actual edges of $G_n$, the one remaining pair of consecutive vertices is from $\mathcal{Q}$, and the distance between the vertices in this pair is at least $\ell+2$ in $G_n$.
\end{itemize}

Let us recall that whp $G_n$ has diameter $k+2$ due to Corollary~\ref{cor:diameter}. Assuming that the pair $\{u_1,u_2\}$ is $\ell$-undetectable for all $\ell\in[k]$, observe that whp the addition of the edge $\{u_1,u_2\}$ agrees with the answers to all the queries. Indeed, if it affects some queried pair $\{w,w'\}$, then there should be a path $w\ldots u_1u_2\ldots w'$ of length $\ell+1\leq k+1$ in $G_n\cup\{u_1,u_2\}$ while in $G_n$ the distance between $w,w'$ should be at least $\ell+2$  --- a contradiction. Thus, whp the event that there exists a pair $\{u_1,u_2\}$ which is $\ell$-undetectable for all $\ell\in[k]$ implies that $G_n$ is not $\mathcal{Q}$-reconstructible. It then remains to prove that whp $G_n$ has a pair $\{u_1,u_2\}$ which is $\ell$-undetectable for all $\ell\in[k]$.

Let us now consider the following iterative procedure. At every step $i$, we are given with a set of ``considered'' vertices $X_i\subset[n]$ (initially $X_1=\varnothing$), then we sample a pair $\{u_1,u_2\}$ uniformly at random from ${[n]\setminus X_i\choose 2}$ and expose $G_n[N_k(u_1)]$ and $G_n[N_k(u_2)]$. If the pair $\{u_1,u_2\}$ is not small or $u_1,u_2$ are adjacent (note that the exposed edges are enough to get these decisions), then we just skip this step and switch to the step $i+1$ with $X_{i+1}:=X_i\cup N_k(u_1)\cup N_k(u_2)$.

Now, assume that at the current step $i$, we observe that the pair $\{u_1,u_2\}$ is small and $u_1,u_2$ are not adjacent. Note that the pair $\{u_1,u_2\}$ is $\ell$-undetectable for every $\ell\in[k-1]$ since it is small (so it cannot have any ``path'' of length at most $k$ from the definition of an undetectable pair). The crucial observation is that the edges between $N_k(u_1)\setminus N_k(u_2)$ and $N_k(u_2)\setminus N_k(u_1)$ have not been exposed yet unless they are entirely inside $X_i$. We then expose these missing edges and show that they define event that has sufficiently large probability and implies that the pair $\{u_1,u_2\}$ is $k$-undetectable (see Claim~\ref{cl:good_pairs} below). After this step, we switch to the step $i+1$ with $X_{i+1}:=X_i\cup N_k(u_1)\cup N_k(u_2)$ as before.

We perform $\tau=\lfloor n^{1/(k+1)-o(1)}\rfloor$ steps, where $\tau$ is chosen so that $\tau(np)^k=o(n/\log n)$. Due to Claim~\ref{cl:balls}, whp (in the measure of the union of the balls exposed before this step) $|X_i|=o(n/\log n)$, implying that the probability that the pair $\{u_1,u_2\}$ we consider at step $i$ is small is at least $1/2-o(1)$ by Claim~\ref{cl:half_pairs_small}. 

\begin{claim}
There exists an event $\mathcal{B}_{u_1,u_2}$ that depends only on edges between $N_k(u_1)\setminus N_k(u_2)$ and $N_k(u_2)\setminus  N_k(u_1)$  such that
$$
\mathbb{P}(\mathcal{B}_{u_1,u_2}\mid G_n[X_i])\geq n^{-7/(8(k+1))+o(1)}
$$
and $\mathcal{B}_{u_1,u_2}$ implies that the pair $\{u_1,u_2\}$ is $k$-undetectable. 
\label{cl:good_pairs}
\end{claim}

\begin{proof}

Fix $k'\leq k$, $v_1\in N_{k'}(u_1)$, $v_2\in N_{k-k'}(u_2)$ such that the pair $\{v_1,v_2\}$ was queried. Note that since $\{u_1,u_2\}$ is small, it may only happen when $v_1\notin N_k(u_2)$, $v_2\notin N_k(u_1)$.

Without loss of generality, assume that $k'\leq k-k'$. Note that, if there exists an edge in $G_n$ between $A(v_1):=N_{k-k'}(v_1)\setminus (X_i\cup N_k(u_2))$ and $A(v_2):=N_{k'}(v_2)\setminus N_k(u_1)$, then the distance between $v_1,v_2$ is at most $k+1$. Due to Claim~\ref{cl:balls}, whp 
\begin{align*}
|N_{k-k'}(v_1)|=(1+o(1/\log n))(np)^{k-k'}, &\quad  |N_k(u_1)|=(1+o(1/\log n))(np)^k,\\
|N_{k'}(v_2)|=(1+o(1/\log n))(np)^{k'}, &\quad |N_k(u_2)|=(1+o(1/\log n))(np)^k.
\end{align*}

Since $|X_i|=o(n/\log n)$, due to Hoeffding's exponential tail bound for the hypergeometric distribution, whp 
$$
|A(v_1)|=(1+o(1/\log n))(np)^{k-k'},\quad
|A(v_2)|=(1+o(1/\log n))(np)^{k'}.
$$
Then
\begin{align*}
 \mathbb{P}(\text{there is an edge between }A(v_1)\text{ and }A(v_2)) & =1-(1-p)^{(1+o(1/\log n))(np)^{k}}\\
 &=1-e^{-\lambda+o(1)}.
\end{align*}
Since $\{u_1,u_2\}$ is small, the number of queried $\{v_1,v_2\}$ in $\bigcup_{k'=0}^k N_{k'}(u_1)\times N_{k-k'}(u_2)$ is at most $M$. Now, we define the desired event $\mathcal{B}_{u_1,u_2}$: for every $k'$ and every queried pair $\{v_1\in N_{k'}(u_1),\,v_2\in N_{k-k'}(u_2)\}$, there exists an edge between the respective $A(v_1)$ and $A(v_2)$. Then, due to the Harris inequality~\cite{Harris},
$$
 \mathbb{P}(\mathcal{B}_{u_1,u_2})\geq (1-e^{-\lambda+o(1)})^{M}\sim \exp\left[-\frac{7+o(1)}{8(k+1)}\ln n\right].
$$

\end{proof}

Whp in at least $(1/2-o(1))\tau$ steps we get a small pair. Due to Claim~\ref{cl:good_pairs}, we then get that the probability that, for at least one of these small pairs $\{u_1,u_2\}$, the respective event $\mathcal{B}_{u_1,u_2}$ happens is at least 
$$
1-\left(1-n^{-7/(8(k+1))+o(1)}\right)^{(1/2-o(1))\tau}=1-\exp\left[-\frac{1}{2}n^{\frac{1+o(1)}{8(k+1)}}\right]=1-o(1),
$$
completing the proof.

\section{Open questions}
\label{sc:further}

Although we get tight bounds on distance query complexity for the binomial random graph $G(n,p)$ when the edge probability $p$ satisfies~\eqref{eq:p_tights_bounds_adaptive}, as $p$ gets closer to a hitting time of a diameter increment, the difference between the bounds increases. In particular, there is a polynomial gap between the bounds even for the ``non-adaptive reconstructibility'' when $p=\Theta((n^{-k}\ln n)^{1/(k+1)})$ satisfies~\eqref{eq:p_NO_tights_bounds}. It would be interesting to get at least the right power of $n$ in the query complexity in this case.

The further natural step is to generalise our results to growing diameter, i.e. to $p=n^{-1+o(1)}$. Although Theorems~\ref{th2},~\ref{th1_2},~and~\ref{th3} are generalised directly to $p\gg \ln^2n/n$, it is unclear whether it is possible to extend Theorem~\ref{th1} without any significant modification of the proof method. The main complication is in the application of the bounded difference inequality (Theorem~\ref{th:McDiarmid}) since we use the fact that $e^{-2i}>n^{-2/(k+1)}$ for all $i< \ln n/C$ and some sufficiently large constant $C$. If $k$ is growing, then there is no such $C$.

As we mention in Introduction, Mathieu and Zhou~\cite{MZ21} proved that, for constant~$d$, the query complexity of the random $d$-regular graph on $[n]$ is at most $n\ln^2n$ whp. On the other hand, it is easy to prove the information--theoretic lower bound $\Omega(n\log n/\log\log n)$. It seems likely that the query complexity of the random $d$-regular graph actually equals $n(\log n)^{1-o(1)}$ whp, and we are asking whether the lower bound is tight.

\end{document}